\documentclass[letterpaper,11pt]{scrartcl}

\usepackage[margin=1in]{geometry}

\usepackage{etex}
\usepackage{scrhack}
\usepackage{cmap}
\usepackage[ngerman,english]{babel}
\usepackage[utf8]{inputenc}
\usepackage{graphicx}
\usepackage{cite}
\usepackage{caption}
\usepackage{subcaption}
\usepackage{epstopdf}
\usepackage{mathtools}
\usepackage{amsmath}
\usepackage{amsthm}
\usepackage{amssymb}
\usepackage{comment}
\usepackage{multirow}
\usepackage{tikz}
\usetikzlibrary{arrows,intersections}
\usetikzlibrary{shapes,snakes}
\usepackage[english,vlined,ruled]{algorithm2e}
\usepackage{hyperref}
\usepackage{lipsum}
\usepackage{float}
\usepackage{tabularx}
\usepackage{booktabs}
\usepackage{longtable}
\usepackage{tabu}
\usepackage{environ}
\usepackage{mdframed}
\usepackage{datetime}
\usepackage{enumerate}
\usepackage{authblk}

\newtheorem{theorem}{Theorem}
\newtheorem{cor}{Corollary} 
\newtheorem{lem}{Lemma}

\theoremstyle{definition}
\newtheorem*{defi}{Definition} 

\theoremstyle{remark}
\newtheorem*{rem}{Remark}

\mdfdefinestyle{probframe}{skipabove=0pt,skipbelow=0pt,innerleftmargin=0pt,
innerrightmargin=5pt,innertopmargin=0pt, innerbottommargin=5pt}

\makeatletter
\newcommand{\problemtitle}[1]{\gdef\@problemtitle{#1}}
\newcommand{\probleminput}[1]{\gdef\@probleminput{#1}}
\newcommand{\problemquestion}[1]{\gdef\@problemquestion{#1}}
\NewEnviron{problem}{
    \medskip
\begin{mdframed}[style=probframe]
  \problemtitle{}\probleminput{}\problemquestion{}
  \BODY
  \par\addvspace{.5\baselineskip}
  \noindent
  \begin{tabularx}{\textwidth}{@{\hspace{\parindent}} l X c}
    \multicolumn{2}{@{\hspace{\parindent}}l}{\@problemtitle} \\
    \textbf{Input:} & \@probleminput \\
    \textbf{Question:} & \@problemquestion
  \end{tabularx}
  \par\addvspace{.5\baselineskip}
\end{mdframed}
    \medskip
}
\makeatother

\newcommand{\ra}{\rightarrow}

\newcommand{\val}{\operatorname{val}}

\newcommand{\nn}{\mathbb{N}}

\newcommand{\rr}{\mathbb{R}}

\newcommand{\udot}{\mathbin{\dot{\cup}}}

\newcommand{\rank}{\operatorname{rk}}

\newcommand{\matroid}{\mathcal{M}}

\newcommand{\basis}{\mathcal{B}}
\newcommand{\inds}{\mathcal{F}}

\title{Matroid Bases with Cardinality Constraints on the Intersection}

    \author[1]{Stefan Lendl\footnote{lendl@math.tugraz.at}}
    \author[2]{Britta Peis\footnote{britta.peis@oms.rwth-aachen.de}}
    \author[2]{Veerle Timmermans\footnote{veerle.timmermans@oms.rwth-aachen.de}}
	
	\affil[1]{Graz University of Technology, Institute of Discrete Mathematics}
	\affil[2]{RWTH Aachen, Chair of Management Science}	
	\date{}

\begin{document}
\maketitle

\begin{abstract} 
Given two matroids $\matroid_{1} = (E, \basis_{1})$ and $\matroid_{2} = (E,
\basis_{2})$ on a common ground set $E$ with base sets $\basis_1$ and
$\basis_2$, some integer $k \in \nn$, and two cost functions $c_{1}, c_{2}
\colon E \ra \rr$, we consider the optimization problem to find a basis $X \in
\basis_{1}$ and a basis $Y \in \basis_{2}$ minimizing  
cost $\sum_{e\in X} c_1(e)+\sum_{e\in Y} c_2(e)$
 subject to either a lower bound constraint 
$|X \cap Y| \le  k$, an upper bound constraint $|X \cap Y| \ge  k$, or an equality constraint $|X \cap Y| =  k$ on the size of the intersection of the two bases $X$ and $Y$.
The problem with lower bound constraint turns out to be a generalization of the Recoverable Robust Matroid problem under interval uncertainty representation for which the question for a strongly polynomial-time algorithm was left as an open question in \cite{hradovich2017recoverable-MST}.

We show that the two problems with lower and upper bound constraints on the size of the intersection  can be reduced to weighted matroid intersection, and thus be solved with a strongly polynomial-time primal-dual algorithm. The question whether the problem with equality constraint can also be solved efficiently turned out to be a lot harder.
As our main result, we present a strongly polynomial, primal-dual algorithm for the problem with equality constraint on the size of the intersection.

Additionally, we discuss generalizations of the problems from matroids to polymatroids, and from two to three or more matroids.

%
\end{abstract}

\section{Introduction}\label{sec:intro}

Matroids are fundamental and well-studied structures in combinatorial optimization. Recall that a matroid $\mathcal{M}$ is a tuple $\mathcal{M}=(E,\mathcal{F})$, consisting of
 a finite ground set $E$ and a family of subsets $\mathcal{F}\subseteq 2^E$, called the \emph{independent sets},
satisfying (i) $\emptyset \in \mathcal{F}$, (ii) if $F\in \mathcal{F}$ and $F'\subset F$, then $F'\in \mathcal{F}$, and (iii) if $F,F'\in \mathcal{F}$ with $|F'|>|F|$, then there exists some element $e\in F'\setminus{F}$ satisfying
$F\cup \{e\}\in \mathcal{F}$. As usual when dealing with matroids, we assume that a matroid is specified via
an independence oracle, that, given $S\subseteq E$ as input, checks whether or not $S\in \mathcal{F}$. 
Any inclusion-wise maximal set in independence system $\mathcal{F}$ is called a \emph{basis} of $\mathcal{F}$.
Note that the set of bases $\mathcal{B}=\mathcal{B}(\mathcal{M})$ of a matroid $\mathcal{M}$ uniquely defines its independence system via $\mathcal{F}(\mathcal{B})=\{F\subseteq E\mid F\subseteq B \mbox{ for some } B\in \mathcal{B}\}.$ 
Because of their rich structure, matroids allow for various different characterizations (see, e.g., \cite{oxley2006matroid}). In particular, matroids can be characterized algorithmically as the only downward-closed structures for which a simple greedy algorithm is guaranteed to return a basis $B\in \mathcal{B}$ of minimum cost $c(B)=\sum_{e\in B} c(e)$ for \emph{any} linear cost function $c:E\to \mathbb{R}_+.$ Moreover, the problem to find a min-cost common base in two matroids $\mathcal{M}_1=(E,\mathcal{B}_1)$ and $\mathcal{M}_2=(E,\mathcal{B}_2)$
 on the same ground set, or the problem to maximize a linear function over the intersection $\mathcal{F}_1\cap \mathcal{F}_2$ of two matroids  $\mathcal{M}_1=(E,\mathcal{F}_1)$ and $\mathcal{M}_2=(E,\mathcal{F}_2)$ can be done efficiently with a strongly-polynomial primal-dual algorithm (cf.\ \cite{frank1981weighted}). Optimization over the intersection of three matroids, however, is easily seen to be NP-hard. See \cite{3-matroids} for most recent progress on approximation results for the latter problem. 

Optimization on matroids, their generalization to polymatroids,   or on the intersection of two (poly-)matroids, capture a wide range of interesting problems. 
In this paper, we introduce and study
yet another variant of matroid-optimization problems. Our problems can be seen as a variant or generalization of matroid intersection: we aim at minimizing the sum of two linear cost functions over two bases chosen from two matroids on the same ground set  with an additional cardinality constraint on the intersection of the two bases. As it turns out, the problems with lower and upper bound constraint are computationally equivalent to matroid intersection, while the problem with equality constraint seems to be strictly more general, and to lie somehow on the boundary of efficiently solvable combinatorial optimization problems. Interestingly, while the problems on matroids with lower, upper, or equality constraint on the intersection can be shown to be solvable in strongly polynomial time,
the extension of the problems towards polymatroids is solvable in strongly polynomial time for the lower bound constraint, but NP-hard for both, upper and equality constraints.

\paragraph{The model}
Given two matroids $\matroid_{1} = (E, \basis_{1})$ and $\matroid_{2} = (E,
\basis_{2})$ on a common ground set $E$ with base sets $\basis_1$ and
$\basis_2$, some integer $k \in \nn$, and two cost functions $c_{1}, c_{2}
\colon E \ra \rr$, we consider the optimization problem to find a basis $X \in
\basis_{1}$ and a basis $Y \in \basis_{2}$ minimizing  $c_{1}(X) + c_{2}(Y)$ subject to either a lower bound constraint 
$|X \cap Y| \ge  k$, an upper bound constraint $|X \cap Y| \le  k$, or an equality constraint $|X \cap Y| =  k$ on the size of the intersection of the two bases $X$ and $Y$.
Here, as usual, we write $c_1(X)=\sum_{e\in X} c_1(e)$ and $c_2(Y)=\sum_{e\in Y} c_2(e)$ to shorten notation.
Let us denote the following problem by $(P_{= k})$.
 \begin{align*}
    \min\ & c_{1}(X) + c_{2}(Y)\\
    \text{s.t. } 
    & X \in \mathcal{B}_{1}\\
    & Y \in \mathcal{B}_{2}\\
    & |X \cap Y| = k
\end{align*}

Accordingly, if $|X \cap Y| = k$ is replaced by either the upper bound constraint $|X \cap Y| \leq k$, or
the lower bound constraint $|X \cap Y| \geq k$,
 the problem is called $(P_{\le k})$ or $(P_{\ge k})$, respectively. Clearly, it only makes sense to consider integers $k$ in the range between $0$ and $K:=\min\{\rank(\matroid_1), \rank(\matroid_2)\}$, where $\rank(\matroid_i)$ for $i\in \{1,2\}$ is the rank of matroid $\matroid_i$, i.e., the cardinality of each basis in $\matroid_i$ which is unique due to the augmentation property (iii). For details on matroids, we refer to \cite{oxley2006matroid}. 
 
 \paragraph{Related Literature on the Recoverable Robust Matroid problem:} Problem $(P_{\ge k})$ in the special case where $\mathcal{B}_1=\mathcal{B}_2$ is known and well-studied under the name \emph{Recoverable Robust Matroid Problem Under Interval Uncertainty Representation}, see \cite{busing2011phd, hradovich2017recoverable,hradovich2017recoverable-MST} and Section \ref{sec.RecRob} below. For this special case of $(P_{\ge k})$,
 B\"using~\cite{busing2011phd} presented an algorithm   which is exponential in $k$. 
In 2017, Hradovich, Kaperski, and Zieli\'{n}ski~\cite{hradovich2017recoverable} proved that the problem can be solved in polynomial time via some iterative relaxation algorithm and asked for a strongly polynomial time algorithm. Shortly after, the same authors  presented
in \cite{hradovich2017recoverable-MST} a strongly polynomial time primal-dual algorithm for the special case of the problem on a graphical matroid. The question whether a strongly polynomial time algorithm for $(P_{\ge k})$ with $\mathcal{B}_1=\mathcal{B}_2$  exists remained open. 

\paragraph{Our contribution.} In Section \ref{sec:reduction}, we show that both, $(P_{\le k})$ and $(P_{\ge k})$, can be polynomially reduced to weighted matroid intersection. Since weighted matroid intersection can be solved in strongly polynomial time by some very elegant primal-dual algorithm~\cite{frank1981weighted}, this answers the open question raised in \cite{hradovich2017recoverable} affirmatively. 

As we can solve $(P_{\le k})$ and $(P_{\ge k})$ in strongly polynomial time via some combinatorial algorithm, the question arises whether or not the problem with equality constraint $(P_{=k})$ can be solved in strongly polynomial time as well. The answer to this question turned out to be more tricky than expected. As our main result,
in Section \ref{sec:equality}, we provide a strongly polynomial, primal-dual  algorithm that constructs an optimal solution for $(P_{=k})$. The same algorithm can also be used to solve an extension, called $(P^{\Delta})$, of the Recoverable Robust Matroid problem under Interval Uncertainty Represenation, see Section \ref{sec.RecRob}.

Then, in Section \ref{sec:polymatroids}, we consider the generalization of  problems $(P_{\le k})$, $(P_{\ge k})$, and $(P_{= k})$ from matroids to polymatroids with  lower, upper, or equality bound constraint, respectively,  on the size of the meet $|x \wedge y| :=\sum_{e\in E} \min\{x_e, y_e\}$. Interestingly, as it turns out, the generalization of $(P_{\ge k})$ can be solved in strongly polynomial time via reduction to some polymatroidal flow problem, while  the generalizations of $(P_{\le k})$ and $(P_{= k})$ can be shown to be weakly NP-hard, already for uniform polymatroids.

Finally, in Section~\ref{sec:multistage} we discuss the generalization of our matroid problems from two to three or more matroids. That is, we consider $n$ matroids $\matroid_i=(E,\basis_i)$, $i\in [n]$, and $n$ linear cost functions $c_i:E\to \mathbb{R}$, for $i\in [n]$. The task is to find $n$ bases $X_i\in \basis_i$, $i\in [n]$, minimizing the cost $\sum_{i=1}^n c_i(X_i)$ subject to   a cardinality constraint on the size of intersection $|\bigcap_{i=1}^n X_i|$. When we are given an upper bound on the size of this intersection we can find an optimal solution within polynomial time. Though, when we have an equality or a lower bound constraint on the size of the intersection, the problem becomes strongly NP-hard.


\section{Reduction of $(P_{\le k})$ and $(P_{\ge k})$ to weighted matroid intersection}
\label{sec:reduction}
We first note that $(P_{\le k})$ and $(P_{\ge k})$ are computationally equivalent.
To see this, consider any instance $(\matroid_1, \matroid_2, k, c_1, c_2)$ of
$(P_{\ge k})$, where $\matroid_1=(E, \basis_1)$, and $\matroid_2=(E, \basis_2)$ are two matroids on the same ground set $E$ with base sets $\basis_1$ and $\basis_2$, respectively. Define $c_2^*=-c_2$, $k^*=\rank(\matroid_1)-k$, and let
$\matroid_2^*=(E, \basis_2^*)$ with $\basis_2^*=\{E\setminus{Y}\mid Y\in \basis_2\}$ be the dual matroid of $\matroid_2$. Since
\begin{itemize}
\item[(i)] $|X\cap Y| \le k \iff |X\cap (E\setminus{Y})| =|X|-|X\cap Y| \ge \rank(\matroid_1)-k=k^*$, and
\item[(ii)] $c_1(X)+c_2(Y)=c_1(X)+c_2(E)-c_2(E\setminus{Y})=
c_1(X)+c^*_2(E\setminus{Y})+c_2(E)$,
\end{itemize}
where $c_2(E)$ is a constant, it follows that 
$(X,Y)$ is a minimizer of $(P_{\ge k})$ for the given instance $(\matroid_1, \matroid_2, k, c_1, c_2)$ if and only if
$(X, E\setminus{Y})$ is a minimizer of $(P_{\le k^*})$ for 
$(\matroid_1, \matroid^*_2, k^*, c_1, c^*_2)$.
Similarly, it can be shown that any problem of type $(P_{\le k})$ polynomially reduces to an instance of type $(P_{\ge k^*})$.

\begin{theorem}
Both problems, $(P_{\le k})$ and $(P_{\ge k})$, can be reduced to weighted matroid intersection.
\end{theorem}

\begin{proof}
By our observation above, it suffices to show that $(P_{\le k})$
can be reduced to weighted matroid intersection.
Let $\tilde{E} := E_{1} \udot E_{2}$, where $E_{1}, E_{2}$ are two copies of our original ground set $E$. We consider $\mathcal{N}_{1} = (\tilde{E}, \tilde{\inds}_{1}), \mathcal{N}_{2} = (\tilde{E}, \tilde{\inds}_{2})$, two special types of matroids on this new ground set $\tilde{E}$, 
where $\inds_{1}, \inds_{2}, \tilde{\inds}_{1}, \tilde{\inds}_{2}$ are the sets of independent sets of $\matroid_{1}, \matroid_{2}, \mathcal{N}_{1}, \mathcal{N}_{2}$ respectively.
Firstly, let
$\mathcal{N}_{1} = (\tilde{E}, \tilde{\inds}_{1})$ be the direct sum of $\matroid_{1}$ on $E_{1}$ and $\matroid_{2}$ on $E_{2}$. That is, for $A \subseteq \tilde{E}$ it holds that $A \in \tilde{\inds}_{1}$ if and only if $A \cap E_{1} \in \inds_{1}$ and $A \cap E_{2} \in \inds_{2}$.

The second matroid $\mathcal{N}_{2} = (\tilde{E}, \tilde{\inds}_{2})$ is defined 
as follows: we call $e_{1} \in E_{1}$ and $e_{2} \in E_{2}$  a pair, if $e_{1}$ and $e_{2}$ are copies of
the same element in $E$. If $e_{1}, e_{2}$ are a pair then we call $e_{2}$ the sibling of $e_{1}$ and vice versa. 
Then
\[\tilde{\inds}_{2} := \{ A \subseteq \tilde{E} \colon A \mbox{ contains at most $k$ pairs}\}.\]
For any $A \subseteq \tilde{E}$, $X=A\cap E_1$ and $Y=A\cap E_2$ forms a feasible solution for
$(P_{\le k})$ if and only if
$A$ is a basis in matroid $\mathcal{N}_1$ and independent in matroid $\mathcal{N}_2.$
Thus, $(P_{\le k})$ is equivalent to the weighted matroid intersection problem
$$\max\{w(A) \colon A\in \tilde{\inds}_{1} \cap \tilde{\inds}_{2}\},$$
with weight function
\[
w(e)=
\begin{cases}
C-c_1(e) & \mbox{ if } e\in E_1,\\
C-c_2(e) & \mbox{ if } e\in E_2,
\end{cases}
\]
for some constant $C>0$ chosen large enough to ensure that $A$ is a basis in $\mathcal{N}_1$.
To see that $\mathcal{N}_2$ is indeed a matroid, we first observe that 
$\tilde{\inds}_{2}$ is non-empty and downward-closed (i.e., $A\in \tilde{\inds}_{2}$, and $B\subset A$ implies $B\in \tilde{\inds}_{2}$).
To see that $\tilde{\inds}_{2}$ satisfies the matroid-characterizing augmentation property
$$A,B\in \tilde{\inds}_{2} \mbox{ with } |A| \le |B| \mbox{ implies } \exists
e\in B\setminus{A} \mbox{ with } A+e \in \tilde{\inds}_{2},$$
take any two independent sets
$A,B\in \tilde{\inds}_{2}$. If $A$ cannot be augmented from $B$, i.e., if
$A+e\not\in \tilde{\inds}_{2}$ for every $e\in B\setminus{A}$, then
$A$ must contain exactly $k$ pairs, and for each  $e\in B\setminus{A}$, the sibling of $e$ must be contained in $A$. This implies $|B| \le |A|$, i.e., $\mathcal{N}_2$ is a matroid.
\end{proof}

Weighted matroid intersection is known to be solvable within strongly polynomial time (e.g., see Frank~\cite{frank1981weighted}). Hence, both $(P_{\leq k})$ and $(P_{\geq k})$ can be solved in strongly polynomial time.

The same result can be obtained by a reduction to independent matching (see
Appendix~\ref{app:independentmatching}), which in bipartite graphs is known to be solvable within
strongly polynomial time as well (see~\cite{iri1976algorithm}). 

\section{A strongly polynomial primal-dual algorithm for $(P_{=k})$} \label{sec:equality}
We saw in the previous section that both problems, $(P_{\le k})$ and $(P_{\ge k})$,
can be solved in strongly polynomial time via a weighted matroid intersection algorithm.
This leads to the question whether we can solve
the problem $(P_{=k})$ with equality constraint on the size of the intersection efficiently as well.

\paragraph{Some challenges.}
At first sight it seems that we could  use the same construction we  used in Section~\ref{sec:reduction} to show that  $(P_{\le k})$ can be reduced to matroid intersection, and simply ask whether there exists a solution $A\subseteq \tilde{E}$ which is a basis in both, $\mathcal{N}_1$ and $\mathcal{N}_2$. Note, however, that a feasible solution to 
$(P_{=k})$ corresponds to a set $A$ which is a basis in $\mathcal{N}_1$ and an independent set in $\mathcal{N}_2$ with exactly $k$ elements, which is not necessarily  a basis in $\mathcal{N}_2$.
An alternative approach would be to consider reductions to more general, still
efficiently solvable, combinatorial optimization problems. When studying this
problem, it turned out that there are several alternative ways of proving that
$(P_{\le k})$ and $(P_{\ge k})$ can be solved in strongly polynomial time. For
example,  via reduction to  independent bipartite matching (see
Appendix~\ref{app:independentmatching}), or to independent path-matching
(see~\cite{cunningham1997optimal}), or to the submodular flow problem. All of these problems generalize matroid intersection and are still solvable in strongly polynomial time. However, we did not find a way to modify one of those reductions to settle our problem $(P_{=k})$.
In Appendix~\ref{app:independentmatching}, we comment shortly on the main difficulties.

\subsection{The algorithm}
In this section, we describe a primal-dual strongly polynomial algorithm for $(P_{=k})$.
Our algorithm can be seen as a generalization of the algorithm presented by
Hradovich et al. in~\cite{hradovich2017recoverable}.
However, the analysis of our algorithm turns out to be much simpler than the one in \cite{hradovich2017recoverable}.


Let us consider the following piecewise linear concave curve $$\mbox{val}{(\lambda)}=\min_{X\in \basis_1, Y\in \basis_2} c_1(X)+c_2(Y)-\lambda |X\cap Y|,$$
which depends on parameter $\lambda\ge 0$.

\begin{figure}
\begin{center}
\begin{tikzpicture}[yscale=0.75,xscale=3,
    thick,
    >=stealth',
    dot/.style = {
      draw,
      fill = white,
      circle,
      inner sep = 0pt,
      minimum size = 4pt
    }
  ]
  \coordinate (O) at (0,0);
  \draw[->] (-0.0, 0.0) -- (4,0) coordinate[label = {below:}] (xmax);
  \draw[->] ( 0.0,-0.0) -- (0,5) coordinate[label = {right:}] (ymax);
  
  \node (y) at (0,3) {};  
  \draw (y) [dotted] -- (y |- O) node[dot, draw=black,dash pattern=, label = {below:$\lambda_0$}] {};
  
  \node (y1) at (2,3) {};  
  \draw (y1) [dotted] -- (y1 |- O) node[dot, draw=black,dash pattern=, label = {below:$\lambda_1$}] {};
  
  \node (y2) at (3,2) {};  
  \draw (y2) [dotted] -- (y2 |- O) node[dot, draw=black,dash pattern=, label = {below:$\lambda_2,\lambda_3$}] {};
  
  \node (y3) at (3.5,0.5) {};  
  \draw (y3) [dotted] -- (y3 |- O) node[dot, draw=black,dash pattern=, label = {below:$\lambda_4$}] {};

  \draw[scale=1,domain=0:4,smooth,variable=\x,red] plot (\x,3);
  \draw[scale=1,domain=0:4,smooth,variable=\x,red!50!yellow!50] plot (\x,5-\x);
  \draw[scale=1,domain=1.6:4,smooth,variable=\x,yellow] plot (\x,8-2*\x);
  \draw[scale=1,domain=2:3.7,smooth,variable=\x,green] plot (\x,11-3*\x);
  \draw[scale=1,domain=2.2:3.7,smooth,variable=\x,blue] plot (\x,14.5-4*\x);
\end{tikzpicture}
\end{center}
\caption{Visualizing $\val(\lambda)$}
\label{fig:vallambda}
\end{figure}
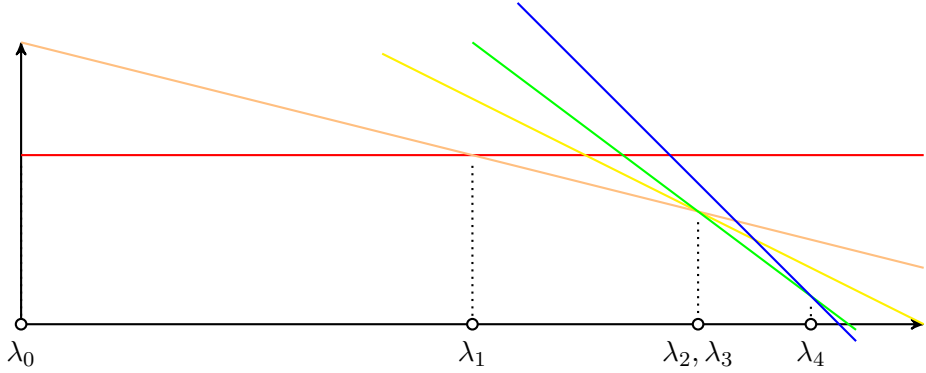

Note that $\val(\lambda)+k \lambda$ is the Lagrangian relaxation of problem $(P_{=k})$. Observe that any base pair $(X,Y)\in \basis_1\times \basis_2$ determines a line
$L_{(X,Y)}(\lambda)$ that hits the vertical axis at $c_1(X)+c_2(Y)$ and has negative slope $|X\cap Y|$. Thus, val$(\lambda)$ is the lower envelope of all such lines. It follows that every base pair $(X,Y)\in \basis_1\times \basis_2$
  which intersects with curve val$(\lambda)$ in either a segment or a breakpoint, and with $|X\cap Y|=k$, is a minimizer of $(P_{=k})$.
  
\paragraph{Sketch of our algorithm.} We first solve the problem
$$\min\{c_1(X)+c_2(Y)\mid X\in \basis_1, \ Y\in \basis_2\},$$
without any constraint on the intersection. Note that this problem can be solved with a matroid greedy algorithm. Let $(\bar{X}, \bar{Y})$ be an optimal solution of this problem.
\begin{enumerate}
\item
If $|\bar{X}\cap \bar{Y}|=k$, we are done as $(\bar{X}, \bar{Y})$ is optimal for 
$(P_{=k})$.
\item Else, if $|\bar{X}\cap \bar{Y}|=k'<k$, our algorithm  starts with the optimal solution $(\bar{X}, \bar{Y})$ for $(P_{=k'})$, and iteratively increases $k'$ by one until $k'=k$. Our algorithm maintains as invariant an optimal solution $(\bar{X}, \bar{Y})$ for the current problem $(P_{=k'})$, together with some dual optimal solution $(\bar{\alpha}, \bar{\beta})$ satisfying the optimality conditions, stated in Theorem \ref{thm:suffpairopt} below, for the current breakpoint $\bar{\lambda}$. Details of the algorithm are described below.
\item Else, if $|\bar{X}\cap \bar{Y}|>k$, we instead consider  an instance of $(P_{=k^*})$ for
$k^*=\rank(\matroid_1)-k$, costs $c_1$ and $c_2^*=-c_2$, and the two matroids
$\matroid_1=(E, \basis_1)$ and $\matroid_2^*=(E, \basis_2^*)$. As seen above, an optimal solution $(X, E\setminus{Y})$ of  problem $(P_{=k^*})$ corresponds to an optimal solution $(X, Y)$ of our original problem $(P_{=k})$, and vice versa. Moreover, $|\bar{X}\cap \bar{Y}|>k$ for the initial base pair $(\bar{X}, \bar{Y})$ implies that
$|\bar{X}\cap (E\setminus{\bar{Y}})|=|\bar{X}|-|\bar{X}\cap \bar{Y}|<k^*$.
Thus, starting with the initial feasible solution $(\bar{X}, E\setminus{\bar{Y}})$ for
$(P_{=k^*})$, we can iteratively
increase $|\bar{X}\cap (E\setminus{\bar{Y}})|$ until $|\bar{X}\cap (E\setminus{\bar{Y}})|=k^*,$ as described in step 2.
\end{enumerate}

Note that a slight modification of the algorithm allows to compute the optimal solutions $(\bar{X}_k, \bar{Y}_k)$ for all $k\in \{0, \ldots, K\}$ in only two runs: run the algorithm for $k=0$ and for $k=K$.

\paragraph{An optimality condition.}
The following optimality condition  turns out to be crucial for the design of our algorithm.

\begin{theorem}[Sufficient pair optimality conditions]\label{thm:suffpairopt}
For fixed $\lambda\ge 0$, base pair
    $(X,Y)\in \basis_1\times \basis_2$ is a minimizer of $\val(\lambda)$ if there exist $\alpha,\beta\in \mathbb{R}^{|E|}_+$ such that 
    \begin{itemize}
        \item[(i)] $X$ is a min cost basis for $c_{1} - \alpha$,
       and $Y$ is a min cost basis for $c_{2} - \beta$;
        \item[(ii)] $\alpha_{e} = 0$ for $e \in X\setminus{Y}$, and
        $\beta_{e} = 0$ for $e \in Y\setminus{X}$;
         \item[(iii)]  $\alpha_{e} + \beta_{e} = \lambda$ for each $ e
    \in E$.
    \end{itemize}
\end{theorem}

The proof of Theorem~\ref{thm:suffpairopt} can be found in Appendix~\ref{app:suffpairopt}.

\paragraph{Construction of the auxiliary digraph.} Given a tuple
$(X,Y,\alpha,\beta, \lambda)$ satisfying the optimality conditions stated in
Theorem \ref{thm:suffpairopt}, we construct a digraph $D=D(X,Y,\alpha,\beta)$
with red-blue colored arcs  as follows (see Figure~\ref{fig:auxgraph}):
\begin{itemize}
\item one vertex for each element in $E$;
\item a red arc $(e,f)$ if $e\not\in X$, $X-f+e \in \basis_1$,
and $c_1(e)-\alpha_e=c_1(f)-\alpha_f$; and
\item a blue arc $(f,g)$ if $g\not\in Y$, $Y-f+g\in \basis_2$, 
and $c_2(g)-\beta_g=c_2(f)-\beta_f.$
\end{itemize}

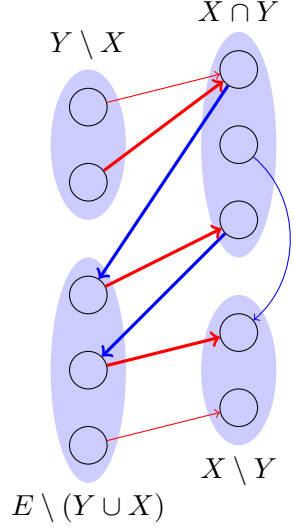
\begin{figure}[ht]
\begin{center}

\begin{tikzpicture}[scale=0.5]
  \node[fill=blue!20,ellipse,minimum width=1cm, minimum height=2cm,label={above:$Y\setminus X$}] (A) at (0,8) {};
  \node[fill=blue!20,ellipse,minimum width=1cm, minimum height=3cm,label={below:$E \setminus (Y\cup X)$}] (B) at (0,2) {};
  \node[fill=blue!20,ellipse,minimum width=1cm, minimum height=3cm,label={above:$X\cap Y$}] (C) at (4,8) {};
  \node[fill=blue!20,ellipse,minimum width=1cm, minimum height=2cm,label={below:$X\setminus Y$}] (D) at (4,2) {};
  \node[draw,circle,minimum size=5mm] (yx1) at (0,9) {};
  \node[draw,circle,minimum size=5mm] (yx2) at (0,7) {};
  \node[draw,circle,minimum size=5mm] (exy1) at (0,4) {};
  \node[draw,circle,minimum size=5mm] (exy2) at (0,2) {};
  \node[draw,circle,minimum size=5mm] (exy3) at (0,0) {};  
  
  \node[draw,circle,minimum size=5mm] (yx21) at (4,10) {};
  \node[draw,circle,minimum size=5mm] (yx22) at (4,8) {};
  \node[draw,circle,minimum size=5mm] (yx23) at (4,6) {};
  \node[draw,circle,minimum size=5mm] (xy31) at (4,3) {};
  \node[draw,circle,minimum size=5mm] (xy32) at (4,1) {};
  
  \draw[red,->] (yx1) -- (yx21);
  \draw[red,very thick, ->] (yx2) -- (yx21);
  
  \draw[blue, ->] (yx22) to [bend left=50] (xy31) ; 
   \draw[blue,very thick, ->] (yx21) -- (exy1);
   \draw[blue,very thick, ->] (yx23) -- (exy2);
  
   \draw[red,very thick,->] (exy1) -- (yx23);
  
  \draw[red,very thick, ->] (exy2) -- (xy31);
  \draw[red,->] (exy3) -- (xy32);
\end{tikzpicture}
\end{center}
\caption{Auxiliary graph constructed in the algorithm for $(P_{=k})$.}
\label{fig:auxgraph}
\end{figure}

\emph{Note:} Although not depicted in Figure~\ref{fig:auxgraph}, there might well be 
blue arcs going  from $Y\setminus{X}$ to either $E\setminus{(X\cup Y)}$ or  $X\setminus{Y}$,
or 
red arcs going from  $Y\setminus{X}$ to $X\setminus{Y}$.

Observe that any red arc $(e,f)$ represents a move in $\mathcal{B}_1$ from $X$ to  $
X \cup \{e\} \setminus \{f\}\in \mathcal{B}_1$. To shorten notation, we write $X \cup \{e\} \setminus \{f\} := X \oplus (e,
f)$. Analogously, any blue arc $(e,f)$ represents a move from $Y$ to
 $Y \cup \{f\} \setminus \{e\}\in \mathcal{B}_2$. Like above, we write $Y \cup \{f\} \setminus \{e\}:=Y\oplus
(e, f)$. 
Given a red-blue alternating path $P$ in $D$ we denote
by $X' = X \oplus P$ the set  obtained from $X$ by performing
all moves corresponding to red arcs, and, accordingly, by
$Y' = Y \oplus P$ the set  obtained from $Y$ by performing
all moves corresponding to blue arcs.

A. Frank proved the following result about sequences of moves, to show correctness of the weighted matroid intersection algorithm.
\begin{lem}[Frank~\cite{frank1981weighted}, {\cite[Lemma 13.35]{korte2007combinatorial}}]
    Let $\matroid = (E, \inds)$ be a matroid, $c \colon E \ra \rr$ and $X \in \inds$. Let $x_{1}, \dots, x_{l} \in X$ and $y_{1}, \dots, y_{l} \notin X$ with
    \begin{enumerate}
        \item $X + y_{j} - x \in \inds$ and $c(x_{j}) = c(y_{j})$ for $j=1,\dots,l$ and
        \item $X + y_{j} -x_{i} \notin \inds$ or $c(x_{i}) > c(y_{j})$ for $1 \leq i,j \leq l$ with $i \neq j$.
    \end{enumerate}
    Then $(X \setminus \{x_{1}, \dots, x_{l}\}) \cup \{y_{1}, \dots, y_{l}\} \in \inds$.
\end{lem}
This suggests the following definition.

\paragraph{Augmenting paths.}
We call any shortest (w.r.t.\ number of arcs) red-blue alternating path linking a vertex in $Y\setminus{X}$ to a vertex in $X\setminus{Y}$ an \emph{augmenting path}.

\begin{lem}\label{lem:franklemma}
If $P$ is an augmenting path in $D$, then 
\begin{itemize}
\item $X'=X\oplus P$ is min cost basis in $\basis_1$ w.r.t.\ costs $c_1-\alpha$,
\item $Y'=Y\oplus P$ is min cost basis in $\basis_2$ w.r.t.\ costs $c_2-\beta$, and
\item $|X'\cap Y'|=|X\cap Y|+1$.
\end{itemize}
\end{lem}
\begin{proof}
    By Lemma~\ref{lem:franklemma}, we know that $X'=X\oplus P$ is min cost basis in $\basis_1$ w.r.t.\ costs $c_1-\alpha$, and $Y'=Y\oplus P$ is min cost basis in $\basis_2$ w.r.t.\ costs $c_2-\beta$. The fact that the intersection is increased by one follows directly by the construction of the digraph.
    \end{proof}
    
\paragraph{Primal update:} Given $(X,Y,\alpha,\beta, \lambda)$ satisfying the optimality conditions and the associated digraph $D$, we update
$(X,Y)$ to $(X', Y')$ with $X'=X\oplus P$, and $Y'=Y\oplus P$, as long as some augmenting path $P$ exists in $D$.
It follows by construction and Lemma~\ref{lem:franklemma} that in each iteration $(X',Y',\alpha, \beta, \lambda)$ satisfies the optimality conditions and that $|X'\cap Y'|=|X\cap Y|+1$.

\paragraph{Dual update:} If $D$ admits no augmenting path, and $|X\cap Y| < k$,  let $R$ denote the set of vertices/elements which are reachable from $Y\setminus{X}$ on some red-blue alternating path. Note that $Y\setminus{X}\subseteq R$ and $(X\setminus{Y})\cap R=\emptyset$.
For each $e\in E$ define the residual costs
$$\bar{c}_1(e):=c_1(e)-\alpha_e, \quad\mbox{ and }\quad \bar{c}_2(e):=c_2(e)-\beta_e.$$
Note that, by optimality of $X$ and $Y$ w.r.t.\ $\bar{c}_1$ and $\bar{c}_2$, respectively,
we have $\bar{c}_1(e)\ge \bar{c}_1(f)$ whenever $X-f+e\in \basis_1$, and
$\bar{c}_2(e)\ge \bar{c}_2(f)$ whenever $Y-f+e\in \basis_2$.

We compute a  "step length" $\delta>0$ as follows: Compute $\delta_1$ and $\delta_2$ via
\[
\delta_1:=\min
\{ \bar{c}_1(e)-\bar{c}_1(f)\mid e\in R\setminus{X},\ f\in X\setminus{R}\colon X-f+e\in \basis_1\},
\]
\[
\delta_2:=\min
\{ \bar{c}_2(g)-\bar{c}_2(f)\mid g\not\in Y \cup R ,\ f\in Y\cap R\colon Y-g+f\in \basis_2\}.
\]

It is possible that the sets over which the minima are calculated are empty.
In these cases we define the corresponding minimum to be $\infty$. In the special
case where $\matroid_{1} = \matroid_{2}$ this case cannot occur.

Since neither a red, nor a blue arc goes from $R$ to $E\setminus{R}$, we know that
both, $\delta_1$ and $\delta_2$, are strictly positive, so that
$\delta:=\min\{\delta_1, \delta_2\}>0$.
Now, update
\[
\alpha'_e=
\begin{cases}
\alpha_e+\delta &\mbox{if }e\in R\\
\alpha_e &\mbox{else.}
\end{cases}
\quad \mbox{ and }\quad
\beta'_e=
\begin{cases}
\beta_e &\mbox{if }e\in R\\
\beta_e +\delta &\mbox{else.}
\end{cases}
\]

\begin{lem}
$(X,Y,\alpha', \beta')$ satisfies the optimality conditions for $\lambda'=\lambda+\delta$.
\end{lem}

\begin{proof}
By construction, we have for each $e\in E$
\begin{itemize}
\item $\alpha_e'+\beta_e'=\alpha_e+\beta_e+\delta = \lambda+\delta=\lambda'$.
\item $\alpha_e'=0$ for $e\in X\setminus{Y}$, since $\alpha_e=0$ and $e \notin R$ (as $(X\setminus{Y})\cap R=\emptyset$).
\item 
$\beta'_e=0$ for $e \in Y \setminus{X}$, since $\beta_e=0$ and $(Y\setminus{X})\subseteq R$.
\end{itemize}

Moreover, by construction and choice of $\delta$, we observe that $X$ and $Y$ are optimal for $c_1-\alpha'$ and $c_2-\beta'$, since
\begin{itemize}
\item
 $c_1(e)-\alpha'_e\ge c_1(f)-\alpha'_f$ whenever $X'-f+e \in \basis_1$, 
\item $c_2(g)-\beta'_g\ge c_2(f)-\beta'_f$ whenever $Y'-f+g \in \basis_2$.
\end{itemize}
To see this, suppose for the sake of contradiction that
 $c_1(e)-\alpha'_e < c_1(f)-\alpha'_f$ for some pair $\{e,f\}$ with $e\not\in X$,  $f\in X$ and $X-f+e \in \basis_1$. 
Then $e\in R$,   $f\not \in R$, 
$\alpha'_e =\alpha_e-\delta$, and $\alpha'_f=\alpha_f,$ implying
$\delta> c_1(e) -\alpha_e- c_1(f) + \alpha_e =\bar{c}_1(e)-\bar{c}_{1}(f),$
in contradiction to our choice of $\delta$.
Similarly, it can be shown that $Y$ is optimal w.r.t.\ $c_2-\beta'$.
Thus, $(X,Y,\alpha',\beta')$ satisfies the optimality conditions for $\lambda'=\lambda+\delta$.
\end{proof}

J. Edmonds proved the following feasibility condition for the non-weighted matroid intersection problem.
\begin{lem}[Edmonds~\cite{edmonds1970submodular}]
    Consider the digraph $\tilde{D} = D(X,Y, 0, 0)$ for cost functions $c_1 = c_2 = 0$ (non-weighted case). If there exists no augmenting path in $\tilde{D}$ then $|X \cap Y|$ is maximum among all $X \in \basis_1, Y \in \basis_2$.
\end{lem}

Based on this result we show the following feasibility condition for our problem.
\begin{lem}\label{lem:edmondsintersection}
    If $\delta = \infty$ and $|X \cap Y| < k$ the given instance is infeasible.
\end{lem}
\begin{proof}
    This follows by the fact that $\delta = \infty$ if and only if the set 
    $(X \setminus Y) \cap R = \emptyset$, even
    if we construct the graph $D'$ without the condition that for red edges $c_{1}(e) - \alpha_{e} = c_{1}(f) - \alpha_{f}$ and for blue edges $c_{2}(g) - \beta_{g} = c_{2}(f) - \beta_{f}$.

    Non existence of such a path implies infeasibility of the instance by Lemma~\ref{lem:edmondsintersection} 
\end{proof}

\begin{lem}
If $(X,Y,\alpha, \beta, \lambda)$ satisfies the optimality conditions and
$\delta < \infty$, a primal update can be performed
after at most  $|E|$ dual updates.
\end{lem}

\begin{proof}
With each dual update, at least one more vertex enters the set $R'$ of reachable elements
in digraph $D'=D(X,Y,\alpha', \beta')$. 
\end{proof}

\paragraph{The primal-dual algorithm.} Summarizing, we obtain the following
algorithm.

\noindent
Input: $\matroid_1=(E,\basis_1)$, $\matroid_2=(E,\basis_2)$, $c_1, c_2:E\to \mathbb{R}$, $k\in \mathbb{N}$\\
Output: Optimal solution $(X,Y)$ of $(P_{=k})$
\begin{enumerate}
\item Compute an optimal solution $(X,Y)$ of $\min\{c_1(X)+c_2(Y)\mid X\in \basis_1, Y\in \basis_2\}$.
\item If $|X\cap Y|=k$, return $(X,Y)$ as optimal solution of $(P_{=k})$.
\item Else, if $|X\cap Y|>k$, run algorithm on  $\matroid_1$, $\matroid^*_2$, $c_1$, $c_2^*:=-c_2$, and $k^*:=\rank(\matroid_1)-k$.
\item Else, set $\lambda:=0$, $\alpha := 0, \beta := 0$.
\item While $|X\cap Y|<k$, do:

\begin{itemize}
\item  Construct auxiliary digraph $D$ based on $(X,Y,\lambda, \alpha, \beta)$ .
\item If there exists an augmenting path $P$ in $D$, update primal
$$X':=X\oplus P, \quad Y':=Y\oplus P.$$
\item Else, compute step length $\delta$ as described above.

    If $\delta = \infty$, terminate with the message "infeasible instance".
    
    Else, set $\lambda:=\lambda+\delta$ and update dual:

\[
\alpha_e :=
\begin{cases}
\alpha_e +\delta & \mbox{ if  $e$ reachable,}\\
\alpha_e & \mbox{ otherwise.}
\end{cases}
\quad
\beta_e :=
\begin{cases}
\beta_e  & \mbox{ if  $e$ reachable,}\\
\beta_e +\delta & \mbox{ otherwise.}
\end{cases}
\]

\item Iterate with $(X,Y,\lambda, \alpha, \beta)$
\end{itemize}
\item Return $(X,Y)$.
\end{enumerate}

As a consequence of our considerations, the following theorem follows.
\begin{theorem}\label{t.1}
The algorithm above solves $(P_{=k})$ using at most $k \times |E|$ primal or dual augmentations.
Moreover, the entire sequence of optimal solutions $(X_k,Y_k)$ for all $(P_{=k})$ with $k=0,1,\dots,K$  can be computed within $|E|^{2}$ primal or dual augmentations.
\end{theorem}

\begin{proof}
    By running the algorithm for $k=0$ and $k=K:=\min\{\rank(\matroid_1), \rank(\matroid_2)\}$ we obtain optimal bases $(X_k,Y_k)$ for $(P_{=k})$ for all $k=1,2,\dots,K$ within $|E|^{2}$ primal or dual augmentations.
\end{proof}

\section{The recoverable robust matroid basis problem -- an application.}\label{sec.RecRob}

There is a strong connection between the model described in this paper and the
recoverable robust matroid basis problem (RecRobMatroid) studied in
\cite{busing2011phd}, \cite{hradovich2017recoverable-MST}, and \cite{hradovich2017recoverable}.
In RecRobMatroid, we are given a matroid $\matroid = (E, \basis)$ on a ground set $E$ with base set $\basis$, some integer $k \in \nn$, a first stage cost function $c_{1}$ and an uncertainty set $\mathcal{U}$ that contains different scenarios $S$, where each scenario $S \in \mathcal{U}$ gives a possible second stage cost $S=(c_S(e))_{e \in E}$.

RecRobMatroid then consists out of two phases: in the first stage, one needs to
pick a basis $X \in \basis$. 
Then, after the actual scenario $S \in \mathcal{U}$ is revealed, there is a second "recovery" stage, where
a second basis $Y$ is picked with the goal to minimize the worst-case cost $c_{1}(X) + c_{S}(Y)$ under the constraint that $Y$ differs in 
at most $k$ elements from the original basis $X$. That is, we require that $Y$ satisfies
$|X\Delta Y| \le k$ or, equivalently, that $|X \cap Y| \geq rk(\matroid) - k$. Here, as usual, $X\Delta Y=X\setminus{Y}\cup Y\setminus{X}$.
The recoverable robust matroid basis problem can be written as follows:

\begin{equation} \label{eq:recrobmatroid}
\min_{X \in \basis} \left( c_1(X) + \max_{S \in \mathcal{U}} \min_{\substack{Y \in \basis \\ |X \cap Y| \geq rk(\matroid) - k} } c_S(Y) \right). 
\end{equation}

There are several ways in which the uncertainty set $\mathcal{U}$  can be represented. One popular way is the \emph{interval uncertainty representation}. In this representation, we are given functions $c': E \rightarrow \rr$, $d: E \rightarrow \rr_+$ and assume that the uncertainty set $\mathcal{U}$ can be represented by a set of $|E|$ intervals:
\[
\mathcal{U} = \left\{ S = (c_S(e))_{e \in E} \mid c_S \in [c'(e), c'(e) + d(e)], e \in E \right\}.
\]
In the worst-case scenario $\bar{S}$ we have for all $e \in E$ that $c_{\bar{S}}(e) =c'(e) + d(e)$. When we define $c_2(e):= c_{\bar{S}}(e)$, it is clear that the RecRobMatroid problem under interval uncertainty represenation (RecRobMatroid-Int, for short) is a special case of $(P_{\ge})$, in which $\basis_1=\basis_2$.

B\"using~\cite{busing2011phd} presented an algorithm for RecRobMatroid-Int  which is exponential in $k$.
In 2017, Hradovich, Kaperski, and Zieli\'{n}ski~\cite{hradovich2017recoverable} proved that RecRobMatroid-Int can be solved in polynomial time via some iterative relaxation algorithm and asked for a strongly polynomial time algorithm. Shortly after, the same authors  presented
in \cite{hradovich2017recoverable-MST} a strongly polynomial time primal-dual algorithm for the special case of RecRobMatroid-Int on a graphical matroid. The question whether a strongly polynomial time algorithm for RecRobMatroid-Int on general matroids exists remained open. 

Furthermore, Hradovich, Kaperski, and Zieli\'{n}ski  showed that when uncertainty
set $\mathcal{U}$ is represented by either budget constraints, or if there is a
bound on the number of elements where scenario costs differ from first stage
costs, the optimal solution to $(P_{\le k})$ is an approximate solution for the problems with these alternative uncertainty sets. These results directly generalize to our model.

\paragraph{Two alternative variants of RecRobMatroid-Int.}
Let us consider two generalizations or variants of RecRobMatroid-Int.
First, instead of setting a bound on the size of the symmetric difference $|X \triangle Y|$ of two bases $X,Y$, alternatively, one could set a penalty on the size of the recovery. Let $C: \nn \rightarrow \rr$ be a 
penalty function which determines the penalty that needs to be paid as a function dependent on the size of the symmetric difference $|X \triangle Y|$. This leads to the following problem, which we denote by
 $(P^{\triangle})$. 
 \begin{align*}
    \min\ & c_{1}(X) + c_{2}(Y) + C(|X \triangle Y|)\\
    \text{s.t. } 
    & X,Y \in \mathcal{B}
\end{align*}
Clearly, $(P^{\triangle})$ is equivalent to RecRobMatroid-Int if $C(|X\triangle Y|)$ is equal to zero as long as $|X\triangle Y|\le k$, and $C(|X\triangle Y|)=\infty$ otherwise. 
As it turns out, our primal-dual algorithm for solving $(P_{=k})$ can be used to efficiently solve $(P^{\triangle})$.

\begin{cor}
Problem $(P^{\triangle})$ can be solved in strongly-polynomial time.
\end{cor}
\begin{proof}
By Theorem \ref{t.1}, optimal solutions $(X_k,Y_k)$ can be computed efficiently for all problems $(P_{=k})_{k\in\{0,1, \ldots, K\}}$  within 
$|E|^{2}$ primal or dual augmentations of the algorithm above.
    It follows that the optimal solution to $(P^{\triangle})$ is a minimizer of
    $$\min\{ c_1(X_k) + c_2 (Y_k) + C(k^{\triangle})\mid k \in \{0, \dots, K\} \},$$ 
    where $k^{\triangle} = \rank(\matroid_1) + \rank(\matroid_2) - 2k$.
\end{proof}

Yet another variant of  RecRobMatroid-Int  or the
more general problem $(P^{\triangle})$ would be to aim for the minimum
\emph{expected} second stage cost, instead of the minimum \emph{worst-case} second stage cost. 
Suppose, with respect to a given probability distribution per element $e\in E$, the expected second stage cost on element $e\in E$ 
is $\mathbb{E}(c_S(e))$.   By the linearity of expectation, to solve these problems, we could  simply solve problem $(P^{\triangle})$ with  $c_2(e):=\mathbb{E}(c_S(e))$.

\section{A generalization to polymatroid base polytopes}
 \label{sec:polymatroids}

 Recall that a function $f:2^E \rightarrow \mathbb{R}$ is called \emph{submodular} if $f(U)+f(V) \geq f(U \cup V) + f(U \cap V)$ for all $U, V \subseteq E$.
Function $f$ is called \emph{monotone} if $f(U) \leq f(V)$ for all $U \subseteq V$, and \emph{normalized} if $f(\emptyset) = 0$. Given a submodular, monotone and normalized function $f$, the pair $(E,f)$ is called a \emph{polymatroid}, and $f$ is called \emph{rank function} of the polymatroid $(E,f)$. 
The associated \emph{polymatroid base polytope} is defined as:
\[ \mathcal{B}(f) := \left\{  x \in \mathbb{\rr}_{+}^{|E|} \mid x(U) \leq f(U) 
\; \forall U \subseteq E, \; x(E) = f(E)\right\},\]
where, as usual,  $x(U) := \sum_{e \in U} x_e$ for all $U\subseteq E$. We refer to the book
"Submodular Functions and
    Optimization" by Fujishige~\cite{fujishige2005submodular} for details on polymatroids and polymatroidal flows as refered to below.

\begin{rem}
We note that all of the arguments presented in this section work also
     for the more general setting of
    \emph{submodular systems} (cf.\ \cite{fujishige2005submodular}), which are defined on arbitrary distributive lattices instead of the Boolean lattice $(2^{E}, \subseteq, \cap, \cup)$. 
    \end{rem}

Whenever $f$ is a submodular function on ground set $E$ with $f(U) \in \nn$ for all $U \subset E$, we call pair $(E,f)$ an \emph{integer polymatroid}.    
Polymatroids generalize matroids in the following sense: if the polymatroid rank function $f$ is integer and additionally satisfies the unit-increase property
$$f(S\cup\{e\})\le f(S)+1\quad \forall S\subseteq E,\ e\in E,$$
then the vertices of the associated polymatroid base polytope $\mathcal{B}(f)$ are exactly the incidence vectors of a matroid
$(E,\mathcal{B})$ with
$\mathcal{B}:=\{B\subseteq E \mid f(B)=f(E)\}$.
Conversely, the rank function $\rank: 2^E\to \mathbb{R}$ which assigns to every
subset $U\subseteq E$ the maximum cardinality $\rank(U)$ of an independent set within $U$ is a polymatroid rank function satisfying the unit-increase property. In particular, bases of a polymatroid base polytope are not necessarily $0-1$ vectors anymore. Generalizing set-theoretic intersection and union from sets (a.k.a.\ $0-1$ vectors) to arbitrary vectors can be done via the following binary operations, called meet and join: given two vectors
$x,y\in \mathbb{R}^{|E|}$ the meet of $x$ and $y$ is
$x\wedge y:=(\min\{x_e,y_e\})_{e\in E}$, and the join of $x$ and $y$ is
$x\vee y:=(\max\{x_e,y_e\})_{e\in E}$. Instead of the size of the intersection, we now talk about the size of the meet, abbreviated by
$$|x\wedge y|:=\sum_{e\in E} \min\{x_e, y_e\}.$$
Similarly, the size of the join is $|x\vee y|:=\sum_{e\in E}\max\{x_e, y_e\}.$
Note that $|x|+|y|=|x\wedge  y|+|x\vee y|$,
where, as usual, for any $x\in \mathbb{R}^{|E|}$, we abbreviate $|x|=\sum_{e\in E}x_e$.
 It follows that for any base pair $(x,y)\in \mathcal{B}(f_1)\times \mathcal{B}(f_2)$, we have 
$$|x|=f_1(E)=\sum_{e\in E\colon x_e>y_e} (x_e-y_e) -|x\wedge y| \text{ and }
|y|=f_2(E)=\sum_{e\in E\colon y_e>x_e} (y_e-x_e) -|x\wedge y|.$$ 
Therefore, it holds that  
$|x\wedge y| \ge k$ if and only if both, 
$\sum_{e\in E \colon x_e > y_e} (x_e-y_e) \le f_1(E)-k$ and 
$\sum_{e\in E \colon x_e < y_e} (y_e-x_e) \le f_2(E)-k$.
The problem described in the next paragraph can be seen as a direct generalization of problem $(P_{\ge k})$ when going from matroid bases  to more general polymatroid base polytopes.

\paragraph{The model.} Let $f_{1}, f_{2}$ be two polymatroid rank functions with
associated polymatroid base polytopes $\mathcal{B}(f_{1})$ and $\mathcal{B}(f_{2})$ defined on the same ground set of resources $E$, let $c_1, c_2:E\to \mathbb{R}$ be two cost functions on $E$, and let $k$ be some integer. The following problem, which we denote by
$(\hat{P}_{\ge k})$, is a direct generalization of $(P_{\ge k})$ from matroids to polymatroids.
\begin{align*}
    \min\ & \sum_{e \in E} c_{1}(e) x(e) + \sum_{e \in E}  c_{2}(e) y(e)\\
    \text{s.t. } 
    & x \in \mathcal{B}(f_{1})\\
    & y \in \mathcal{B}(f_{2})\\
    & |x\wedge y| \geq k\\
\end{align*}

Results obtained for $(P_{\ge k}), (P_{\le k})$ and $(P_{= k})$ directly give us pseudo-polynomial algorithms for $(\hat{P}_{\ge k}), (\hat{P}_{\le k})$ and $(\hat{P}_{= k})$.

\begin{cor}
 If $(E,f_1), (E,f_2)$ are two integer polymatroids, problems $(\hat{P}_{\ge k}), (\hat{P}_{\le k})$ and $(\hat{P}_{= k})$ can be solved within pseudo-polynomial time.  
\end{cor}
\begin{proof}
    Each integer polymatroid can be written as a matroid on a pseudo-polynomial number of resources, namely on $\sum_{e \in E}f(\{e\})$ resources~\cite{helgason1974concept}. Hence, the strongly polynomial time algorithms we derived for problems $(P_{\ge k}), (P_{\le k})$ and $(P_{= k})$ can directly be applied, but now have a pseudo-polynomial running time. 
\end{proof}

In the following, we first show that $(\hat{P}_{\ge k})$ can be reduced to an instance of the
\emph{polymatroidal flow problem}, which is known to be computationally equivalent to a submodular flow problem and can thus be solved in strongly polynomial time.
Afterwards, we show that the two problems $(\hat{P}_{\le k})$ and $(\hat{P}_{= k})$, which can be obtained from $(\hat{P}_{\ge k})$ by replacing constraint
$|x\wedge y|\ge k$ by either $|x\wedge y|\le k$, or $|x\wedge y|= k$, respectively, are weakly NP-hard.

\subsection{Reduction of polymatroid base problem $(\hat{P}_{\ge k})$ to polymatroidal flows.}

The polymatroidal flow problem can be described as follows:
we are given a digraph $G=(V,A)$, arc costs $\gamma:A\to \mathbb{R}$, lower
bounds $l:A\to \mathbb{R}$, and two submodular functions $f^+_v$ and $f^-_v$ for
each vertex $v\in V.$ Function $f_v^+$ is defined on $2^{\delta^{+}(v)}$, the
set of subsets of the set $\delta^+(v)$ of $v$-leaving arcs, while $f_v^-$ is
defined on $2^{\delta^{-}(v)}$, the set of subsets of the set $\delta^-(v)$ of
$v$-entering arcs and 
\begin{align*}
    P(f_v^+) = \left\{x \in \rr^{\delta^{+}(v)} \colon x(U) \leq f_v^+(U) \; \forall
    U \subseteq \delta^{+}(v) \right\},\\
    P(f_v^-) = \left\{x \in \rr^{\delta^{-}(v)} \colon x(U) \leq f_v^-(U) \; \forall
    U \subseteq \delta^{-}(v) \right\}.
\end{align*}
Given a flow $\varphi:A\to \mathbb{R},$ the net-flow at $v$ is abbreviated by
$\partial \varphi(v):=\sum_{a\in \delta^-(v)} \varphi(a)-\sum_{a\in \delta^+(v)}
\varphi(a)$. For a set of arcs $S \subseteq A$, $\varphi|_{S}$ denotes the vector
$(\varphi(a))_{a \in S}$. The associated polymatroidal flow problem can now be formulated as follows.
\begin{align*}
    \min & \sum_{a \in A} \gamma(a) \varphi(a) &\\
    \text{s.t. } & l(a) \leq \varphi(a) & (a \in A)\\
    & \partial \varphi(v) = 0 & (v \in V)\\
    & \varphi|_{\delta^{+}(v)} \in P(f^{+}_{v}) & (v \in V)\\
    & \varphi|_{\delta^{-}(v)} \in P(f^{-}_{v}) & (v \in V)\\
\end{align*}
As described in Fujishige's book (see~\cite{fujishige2005submodular}, page 127f), the polymatroidal flow problem is computationally equivalent to submodular flows and can thus be solved in strongly polynomial time.

\begin{theorem}
    The Recoverable Polymatroid Basis Problem can be reduced to the
    Polymatroidal Flow Problem.
\end{theorem}
\begin{proof}
    We create the instance $(G, \gamma, l, (f_{v}^{+})_{v \in V},
    (f_{v}^{-})_{v \in V})$ of the Polymatroid Flow Problem shown in
    Figure~\ref{fig:red-polyflow}. The graph $G$ consists of $6$ special
    vertices $s, u_{1}, u_{2}, v_{1}, v_{2}, t$ and of $12n$ additional vertices
    denoted by $u^{X}_{e}, v^{X}_{e}, u^{Y}_{e}, v^{Y}_{e}, u^{Z}_{e},
    v^{Z}_{e}$ for each $e \in E$. The arc set consists of arcs $(s,u_{1})$,
    $(s, u_{2})$, $(v_{1}, t)$, $(v_{2},t)$, $(t,s)$. In addition we have arcs
    $(u_{1}, u^{X}_{e}), (u_{1}, u^{Z}_{e})$ for each $e \in E$, $(u_{2},
    u^{Y}_{e})$ for each $e \in E$, $(v^{X}_{e}, v_{1})$ for each $e \in E$,
    $(v^{Z}, v_{2}), v^{Y}_{e}, v_{2})$ for each $e \in E$. In addition for each
    $e \in E$ we have three sets of arcs $(u^{X}_{e},
    v^{X}_{e})$, $(u^{Y}_{e}, v^{Y}_{e})$, $(u^{Z}_{e}, v^{Z}_{e})$ which we
    denote by $E_{X}, E_{Y}, E_{Z}$ respectively.  We set $\gamma((u^{X}_{e}, v^{X}_{e})) := c_{1}(e)$,
    $\gamma((u^{Y}_{e}, v^{Y}_{e})) := c_{2}(e)$, $\gamma((u^{Z}_{e}, v^{Z}_{e}))
    := c_{1}(e) + c_{2}(e)$ and $\gamma(a) := 0$ for all other arcs $a$.
    We enforce lower bounds on the flow on the two arcs by setting $l((s,u_{1})) :=
    f_{1}(E)$ and $l((v_{2}, t)) := f_{2}(E)$ and $l(a)=0$ for all other arcs $a$. To model upper bounds on the flow
    along the arcs $(v_{1}, t)$ and $(s,u_{2})$ we add the polymatroidal
    constraints on $\varphi|_{\delta^{+}}(v_{1})$ and
    $\varphi|_{\delta^{-}}(u_{2})$ and set $f_{v_{1}}^{+}((v_{1}, t)) := f_{1}(E)
    - k$ and $f_{u_{2}}^{-}( (s,u_{2}) ) := f_{2}(E) - k$.
    We also set 
    \begin{align*}
        f_{u_{1}}^{+}(S) := f_{1}(\{e \in E \colon (u_{1}, u^{X}_{e}) \in S
        \text{ or } (u_{1}, u^{Z}_{e}) \in S\}) \quad \forall S \subseteq
        \delta^{+}(u_{1}),\\
        f_{v_{2}}^{-}(S) := f_{2}(\{e \in E \colon (v^{Y}_{e}, v_{2}) \in S
        \text{ or } (v^{Z}_{e}, v_{2}) \in S\}) \quad \forall S \subseteq
        \delta^{-}(v_{2}).
    \end{align*}
    All other polymatroidal constraints are set to the trivial polymatroid,
    hence arbitrary in- and outflows are allowed.

    \begin{figure}[t]
      \centering
      \includegraphics[width=0.9\textwidth]{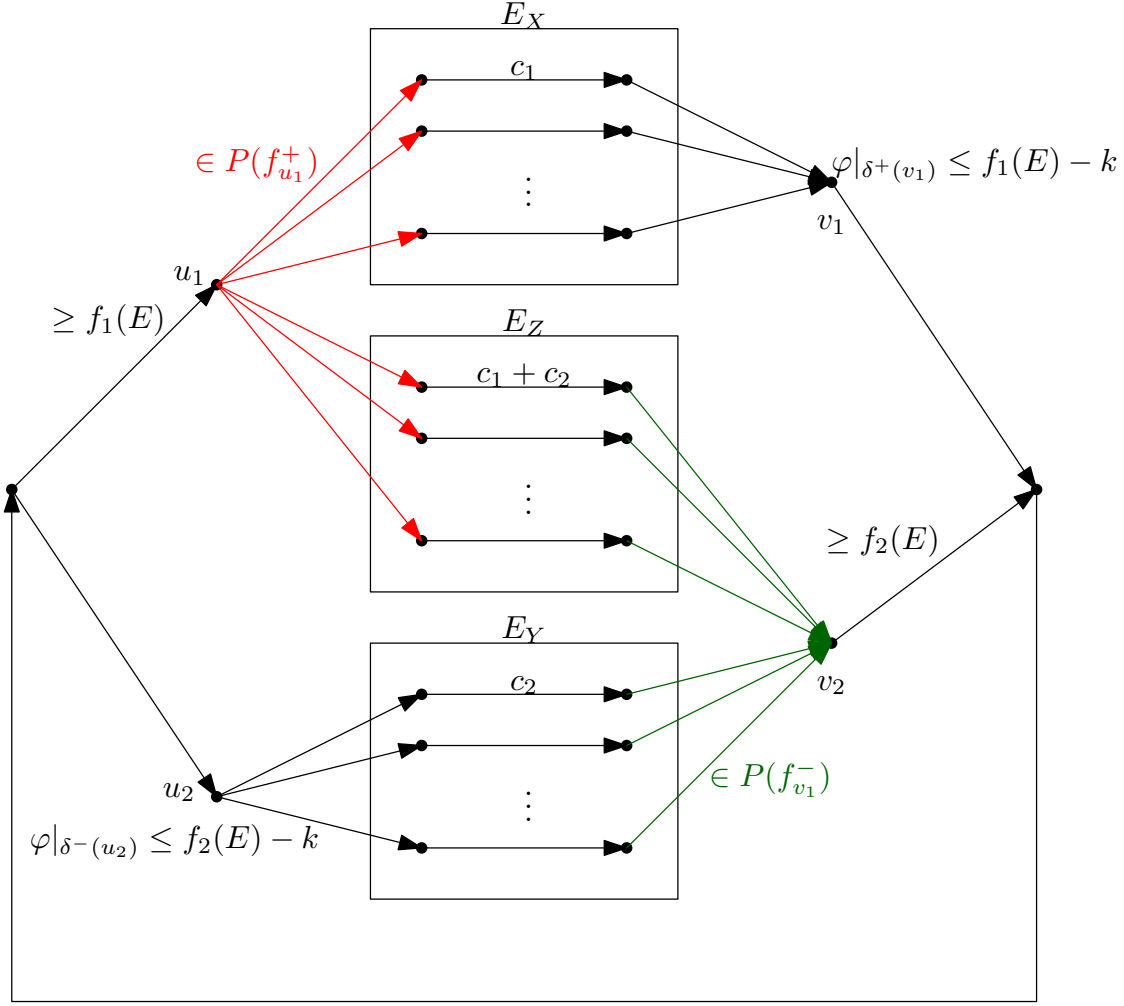}
      \caption{The Polymatroid Flow instance used to solve the Recoverable
      Polymatroid Basis Problem.}
      \label{fig:red-polyflow}
    \end{figure}

    We show that the constructed instance of the polymatroidal flow problem
    is equivalent to the given instance of the Polymatroidal Flow Problem.

    Consider the two designated vertices $u_1$ and $v_2$ such that
    $\delta^+(u_1)$ are the red arcs, and $\delta^-(v_2)$ are the green arcs in
    Figure~\ref{fig:red-polyflow}.
    Take any feasible polymatroidal flow $\varphi$ and let $\tilde{x} :=
    \varphi|_{\delta^{+}(u_{1})}$ denote the restriction of $\varphi$ to the red
    arcs, and $\tilde{y} := \varphi|_{\delta^{-}(v_{2})}$ denote the restriction of $\varphi$ to the green arcs. Note that there is a unique arc entering $u_1$ which we denote by $(s,u_1)$.
    Observe that the constraints $\varphi(s,u_1)\geq f_{1}(E)$ and
    $\varphi|_{\delta^{+}(u_1)}\in P({f}_{u_{1}}^{+})$ for the flow
    going into $E_{X}$ and $E_{Z}$ imply that the flow vector $\tilde{x}$ on the red
    arcs belongs to $ \basis({f}_{u_{1}}^{+})$. Analogously, the flow
    vector $\tilde{y}$ satisfies $\tilde{y} \in \basis({f}_{v_{2}}^{-})$. By
    setting $x(e) := \tilde{x}( (u_{1}, u_{e}^{X}) ) + \tilde{x}( (u_{1},
    u_{e}^{Z}) )$ and 
    $y(e) := \tilde{y}( (v_{e}^{Y}, v_{2}) ) + \tilde{y}( (v_{e}^{Z}, v_{2}) )$
    for each $e \in E$, 
    we have that the cost of the polymatroid flow can be rewritten as
    \[ \sum_{e \in E} c_{1}(e) x(e) + \sum_{e \in E}  c_{2}(e) y(e). \]
    The constraint $\varphi|_{\delta^{+}(u_2)}\le f_2(E)-k$ on the inflow into
    $E_{Y}$, and  the constraint $\varphi|_{\delta^{-}(v_1)}\le f_1(E)-k$ on the
    outflow out of $E_{X}$ are equivalent to $|x\wedge y| \geq k$. Hence, the equivalence follows.

    
\end{proof}

Note that $(\hat{P}_{\geq k})$ is computationally equivalent to
\begin{align*}
    \min\ & \sum_{e \in E} c_{1}(e) x(e) + \sum_{e \in E}  c_{2}(e) y(e)\\
    \text{s.t. } 
    & x \in \mathcal{B}(f_{1})\\
    & y \in \mathcal{B}(f_{2})\\
    & \|x - y\|_{1} \leq k'\\
\end{align*}
which we denote by $(\hat{P}_{\|\cdot\|_{1}})$, as of the direct connection 
$|x|+|y|=2 |x\wedge y| +\|x-y\|_1$
between $|x \wedge y|$, the size of the meet of $x$ and $y$, and the $1$-norm of $x-y$. It is an interesting open question whether this problem is also tractable if one replaces $\|x-y\|_{1} \leq k'$ by arbitrary norms or, specifically, the $2$-norm. We conjecture that methods based on convex optimization could work in this case, likely leading to a polynomial, but not strongly polynomial, running time.

\subsection{Hardness of polymatroid basis problems $(\hat{P}_{\le k})$ and $(\hat{P}_{= k})$}

Let us consider the decision problem associated to problem $(\hat{P}_{\le k})$ 
which can be formulated as follows:
given an instance $(f_1,f_2, c_1, c_2, k)$ of $(\hat{P}_{\le k})$ together with some target value $T\in \mathbb{R}$, decide whether or not
there exists a base pair $(x,y)\in \mathcal{B}(f_1)\times \mathcal{B}(f_2)$ with $|x\wedge y|\le k$ of cost $c_1^Tx+c_2^Ty \le T.$
Clearly, this decision problem belongs to the complexity class NP, since we can verify in polynomial time whether a given pair $(x,y)$ of vectors satisfies the following three conditions (i) $|x\wedge y| \le k$, (ii) $c_1^Tx+c_2^Ty\le T$, and (iii) $(x,y)\in \mathcal{B}(f_1)\times \mathcal{B}(f_2)$.
To verify (iii), we assume, as usual,
the existence of an evaluation oracle. 

\paragraph{Reduction from \textsc{partition}.}
To show that the problem is NP-complete, we show that any instance of the NP-complete problem \textsc{partition}
can be polynomially reduced to an instance of $(\hat{P}_{\le k})$-decision.
Recall the problem \textsc{partition}: given a set $E$ of $n$ real numbers $a_1, \ldots, a_n$, the task is to decide whether or not the $n$ numbers
can be  partitioned  into two sets $L$ and $R$ with $E=L\cup R$
and $L\cap R=\emptyset$ such that $\sum_{j\in L}a_j=\sum_{j\in R}a_j$.

Given an instance $a_1, \ldots, a_n$ of partition with $B:=\sum_{j\in E} a_j$, we construct the following polymatroid rank function
$$f(U)=\min \left\{\sum_{j\in U} a_j, \frac{B}{2}\right\} \quad \forall U\subseteq E.$$
It is not hard to see that $f$ is indeed a polymatroid rank function as it is normalized, monotone, and submodular.
Moreover, we observe that an instance of \textsc{partition} $a_1, \ldots, a_n$ is a yes-instance if and only if
there exist two bases $x$ and $y$ in polymatroid base polytope $\mathcal{B}(f)$ satisfying $|x\wedge y| \le 0.$

Similarly, it can be shown that any instance of  \textsc{partition} can be reduced to  an instance of the decision problem associated to $(\hat{P}_{= k})$,
since an instance of \textsc{partition} is a yes-instance if and only if for the polymatroid rank function $f$ as constructed above
there exists two bases $x$ and $y$ in polymatroid base polytope $\mathcal{B}(f)$ satisfying $|x\wedge y| = 0.$

\section{More than two matroids}\label{sec:multistage}

Another straightforward generalization of the matroid problems $(P_{\le k})$, $(P_{\ge k})$, and $(P_{= k})$ is to
consider more than two matroids, and a constraint on the intersection of the
bases of those  matroids. Given matroids $\matroid_{1} = (E, \basis_{1}),
\dots, \matroid_{M} = (E, \basis_{M})$ on a common ground set $E$, some integer
$k \in \nn$, and cost functions $c_{1}, \dots, c_{M} \colon E \ra \rr$, 
    we consider the optimization problem $(P_{\leq k}^{M})$
 \begin{align*}
     \min\ & \sum_{i=1}^{M}c_{i}(X_{i})\\
    \text{s.t. } 
    & X_{i} \in \mathcal{B}_{i} \quad \forall i =1,\dots,M\\
    & \Big|\bigcap_{i=1}^{M} X_{i}\Big| \leq k.
\end{align*}
Analogously, we define the problems $(P_{\geq k}^{M})$ and $(P_{= k}^{M})$ 
by replacing $\leq k$ by $\geq k$ and $=k$ respectively.

It is easy to observe that both variants $(P_{\geq k}^{M})$ and $(P_{= k}^{M})$
are NP-hard already for the case $M=3$, since even for the feasibility question
there is an easy reduction from the matroid intersection problem for three
matroids.

Interestingly, this is different for $(P^{M}_{\leq k})$. A direct generalization
of the reduction for $(P_{\leq k})$ to weighted matroid intersection (for two matroids)
shown in Section~\ref{sec:reduction} works again.

\begin{theorem}
    Problem $(P^{M}_{\leq k})$ can be reduced to weighted matroid intersection.
\end{theorem}
\begin{proof}
    Let $\tilde{E} := E_{1} \udot \dots \udot E_{M}$, where $E_{1}, \dots, E_{M}$ are $M$ copies of our original ground set $E$. We consider
     $\mathcal{N}_{1} = (\tilde{E}, \tilde{\inds}_{1}), \dots, \mathcal{N}_{2} = (\tilde{E}, \tilde{\inds}_{2})$, two special types of matroids on this new ground set $\tilde{E}$, 
where $\inds_{1}, \dots, \inds_{M}, \tilde{\inds}_{1}, \tilde{\inds}_{2}$ are
the sets of independent sets of $\matroid_{1}, \dots, \matroid_{M}, \mathcal{N}_{1}, \mathcal{N}_{2}$ respectively.
Firstly, let
$\mathcal{N}_{1} = (\tilde{E}, \tilde{\inds}_{1})$ be the direct sum of
$\matroid_{1}$ on $E_{1}$ to $\matroid_{M}$ on $E_{M}$. That is, for $A
\subseteq \tilde{E}$ it holds that $A \in \tilde{\inds}_{1}$ if and only if $A
\cap E_{i} \in \inds_{i}$ for all $i=1,\dots,M$.

The second matroid $\mathcal{N}_{2} = (\tilde{E}, \tilde{\inds}_{2})$ is defined 
as follows: we call $e_{1} \in E_{1}, \dots, e_{M} \in E_{M}$ a line, if $e_{1}$
to $e_{M}$ are copies of the same element in $E$. If $e_{i}$ and $e_{i'}$ are
part of the same line then we call $e_{i}$ a sibling of $e_{i'}$ and vice versa. 
Then
\[\tilde{\inds}_{2} := \{ A \subseteq \tilde{E} \colon A \mbox{ contains at most
$k$ lines}\}.\]
For any $A \subseteq \tilde{E}$, $X_{i}=A\cap E_{i}$ for all $i=1,\dots,M$ forms a feasible solution for
$(P^{M}_{\le k})$ if and only if
$A$ is a basis in matroid $\mathcal{N}_1$ and independent in matroid $\mathcal{N}_2.$
Thus, $(P^{M}_{\le k})$ is equivalent to the weighted matroid intersection problem
$$\max\{w(A) \colon A\in \tilde{\inds}_{1} \cap \tilde{\inds}_{2}\},$$
with weight function
\[
w(e)=
\begin{cases}
    C-c_i(e) & \mbox{ if } e\in E_i, i \in \{1,\dots,M\},\\
\end{cases}
\]
for some constant $C>0$ chosen large enough to ensure that $A$ is a basis in $\mathcal{N}_1$.
To see that $\mathcal{N}_2$ is indeed a matroid, we first observe that 
$\tilde{\inds}_{2}$ is non-empty and downward-closed (i.e., $A\in \tilde{\inds}_{2}$, and $B\subset A$ implies $B\in \tilde{\inds}_{2}$).
To see that $\tilde{\inds}_{2}$ satisfies the matroid-characterizing augmentation property
$$A,B\in \tilde{\inds}_{2} \mbox{ with } |A| \le |B| \mbox{ implies } \exists
e\in B\setminus{A} \mbox{ with } A+e \in \tilde{\inds}_{2},$$
take any two independent sets
$A,B\in \tilde{\inds}_{2}$. If $A$ cannot be augmented from $B$, i.e., if
$A+e\not\in \tilde{\inds}_{2}$ for every $e\in B\setminus{A}$, then
$A$ must contain exactly $k$ lines, and for each  $e\in B\setminus{A}$, the $M-1$ siblings of $e$ must be contained in $A$. This implies $|B| \le |A|$, i.e., $\mathcal{N}_2$ is a matroid.
\end{proof}


\paragraph{Acknowledgement.}  We would like to thank Andr\'as Frank, Jannik 
Matuschke, Tom McCormick, Rico Zenklusen, Satoru Iwata, Mohit Singh, 
Michel Goemans, and Guyla Pap for fruitful discussions at the HIM 
workshop in Bonn, and at the workshop on Combinatorial Optimization in 
Oberwolfach. We would also like to thank
Björn Tauer and Thomas Lachmann for several 
helpful discussions about this topic.

Stefan Lendl acknowledges the support of the Austrian Science Fund 
(FWF): W1230.

\bibliographystyle{plain}
\bibliography{robust}

\appendix

\section{Reduction from $(P_{\geq k})$ to independent bipartite matching}
\label{app:independentmatching}

As mentioned in the introduction, one can also solve problem 
$(P_{\ge k})$ (and, hence, also problem $(P_{\le k})$) by reduction to a special case of the basic path-matching
problem~\cite{cunningham1997optimal}, which is also known under the name
\emph{independent bipartite matching problem}~\cite{iri1976algorithm}. The basic path-matching problem can be
solved in strongly polynomial time, since it is a special case of the submodular
flow problem.

\begin{defi}[Independent bipatite matching]
We are given two matroids $\matroid_1=(E_1,\basis_1)$ and
$\matroid_2=(E_2,\basis_2)$ and a weighted, bipartite graph $G$ on node sets
$E_1$ and $E_2$. The task is to find a minimum weight matching in $G$ such that the elements that are matched in $E_1$ form a basis in $\matroid_1$ and the elements matched in $E_2$ form a basis in $\matroid_2$.
\end{defi}

Consider an instance of $(P_{\geq k})$, where we are given two matroids $\matroid_1, \matroid_2$ on common ground set $E$. We create an instance for the independent matching problem as follows:
Define a bipartite graph $G=(E_{1}, E_{2}, A)$, where $E_{1}$
contains a distinct copy of $E$ and an additional set $U_{1}$ of exactly
$k_2:=\rank_{2}(E)-k$ elements. Let $\tilde{\matroid}_{1}$ be the sum of
$\matroid_{1}$, on the copy of $E$ in $E_{1}$ and the unrestricted matroid of
all subsets of $U_{1}$. The set $E_{2}$ and the matroid $\tilde{\matroid}_{2}$
are defined symmetrically and $U = U_{1} \cup U_{2}$. The set $A$ of edges in $G$ contains an edge between
$\{e, e'\}$ if $e, e'$ are copies of the same element of $E$ in $E_{1}, E_{2}$.
In addition we add all edges $\{e, u\}$ if $e$ is a copy of some element of $E$
in $E_{1}$ and $u \in U_{2}$ or if $e$ is a copy of some element of $E$ in
$E_{2}$ and $u \in U_{1}$. See Figure~\ref{fig:independentmatching} for an
illustration of the constructed bipartite graph.

Observe, that every feasible solution to the independent bipartite matching instance
matches a basis of $\tilde{\matroid}_{1}$ with a basis of
$\tilde{\matroid}_{2}$. Respectively for $i=1,2$, these bases consist of a basis in
$\matroid_{i}$ and all elements in $U_{i}$. Hence at most $\rank_{i}(E) - k$ can
be matched with an arbitrary element of $U$, so at least $k$ elements need to be
matched using edges $\{e, e'\}$. This implies that the size of the intersection of the
corresponding bases in $\matroid_{1}, \matroid_{2}$ is at least $k$.

Let us shortly comment on the main challenge when trying to use this or a similar construction to solve the problem $(P_{=k})$ with equality constraint on the size of the intersection. It is, off course, possible to add additional constraints to enforce that \emph{exactly} $k$ edges matching two copies of the same element appear in any feasible solution to the path-matching problem. However, how can you ensure that the solution does not use two edges $\{e,u\}$ and $\{e, u'\}$ where $e$  corresponds to a copy of the same element of $E$ in both, $E_1$ and $E_2$, and $u\in U_2$ and $u'\in U_1$?

\begin{figure}[ht]
\begin{center}
\begin{tikzpicture}[shorten >=1pt,->]
  \tikzstyle{vertex}=[circle,fill=black!25,minimum size=17pt,inner sep=0pt]

    \node[vertex] (U1-1) at (1,0) {$u_1$};
    \node[vertex] (U2-1) at (7,-3) {$u'_{1}$};
     \node[vertex] (U1-2) at (2,0) {$u_{2}$};
    \node[vertex] (U2-2) at (8,-3) {$u'_{2}$};
     \node[vertex] (U1-3) at (3,0) {$\dots$};
    \node[vertex] (U2-3) at (9,-3) {$\dots$};
     \node[vertex] (U1-4) at (4,0) {$u_{k_2}$};
    \node[vertex] (U2-4) at (10,-3) {$u_{k_1}$};

  \foreach \x in {1,2,3,4}{
    \node[vertex] (M1-\x) at (\x+4,0) {$e_{\x}$};
    \node[vertex] (M2-\x) at (\x,-3) {$e'_{\x}$};
    \draw[thick] (M1-\x) -- (M2-\x);
    }
    
     \node[vertex] (M1-5) at (9,0) {$\dots$};
    \node[vertex] (M2-5) at (5,-3) {$\dots$};
    \draw[thick] (M1-5) -- (M2-5);
    
     \node[vertex] (M1-6) at (10,0) {$e_m$};
    \node[vertex] (M2-6) at (6,-3) {$e'_m$};
    \draw[thick] (M1-6) -- (M2-6);

\foreach \x in {1,2,3,4}{
    \foreach \y in {1,2,3,4,5,6}{
    \draw[black!20] (U1-\x) -- (M2-\y); 
    \draw[black!20] (M1-\y) -- (U2-\x); 
}}
\end{tikzpicture}
\end{center}
\caption{Bipartite graph used in the reduction from $P_{\geq k}$ to basic path-matching}
\label{fig:independentmatching}
\end{figure}
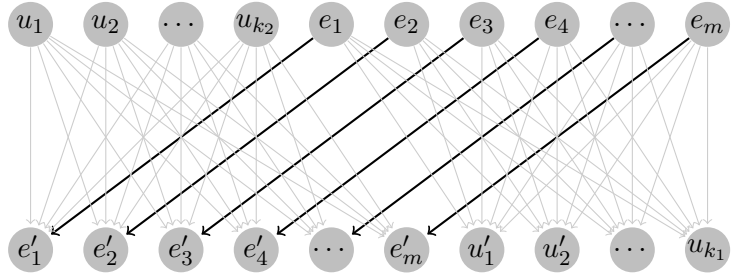

\section{Proof of Theorem~\ref{thm:suffpairopt}}
\label{app:suffpairopt}
Consider the following linear relaxation $(P_{\lambda})$ of an integer programming formulation of problem val$(\lambda)$.
The letters in squared brackets indicate the associated dual variables.
\begin{align*}
    \val(\lambda) = \min\ & \sum_{e \in E} c_{1}(e) x_{e} + \sum_{e \in E}
    c_{2}(e) y_{e} - \lambda \sum_{e \in E} z_{e}  & & \\
    \text{s.t. } & \sum_{e \in E} x_{e} = \rank_{1}(E) & & \quad [\mu]\\
    & \sum_{e \in U} x_{e} \leq \rank_{1}(U) & \forall U \subset E & \quad [w_{U}] \\
    & \sum_{e \in E} y_{e} = \rank_{2}(E) & & \quad  [\nu] \\
    & \sum_{e \in U} y_{e} \leq \rank_{2}(U) & \forall U \subset E & \quad [v_{U}] \\
    & x_{e} - z_{e} \geq 0 & \forall e \in E & \quad [\alpha_{e}] \\
    & y_{e} - z_{e} \geq 0 & \forall e \in E & \quad [\beta_{e}] \\
    & x_{e}, y_{e}, z_{e} \geq 0 &\forall e\in E. &  
\end{align*}

The dual program is then $(D_{\lambda})$:
\begin{align*}
     \max\ & \sum_{U \subset E} \rank_{1}(U) w_{U} +
    \rank_{1}(E) \mu + \sum_{U \subset E} \rank_{2}(U) v_{U} + \rank_{2}(E)
    \nu & \\
    \text{s.t. } & \sum_{U \subset E \colon e \in U} w_{U} + \mu \leq
    c_{1}(e) - \alpha_{e} & \forall e \in E \\
    & \sum_{U \subset E \colon e \in U} v_{U} + \nu \leq
    c_{2}(e) - \beta_{e} & \forall e \in E \\
    & \alpha_{e} + \beta_{e} \geq \lambda & \forall e \in E & \\
    & w_{U}, v_{U} \leq 0. & \forall U \subset E & \\
    & \alpha_{e}, \beta_{e} \geq 0. & \forall e \in E & \\
\end{align*}

Applying the strong LP-duality  to  the two inner problems which correspond to dual variables $(w,\mu)$ and $(v,\nu)$, respectively,  yields

\begin{multline}\max_{\alpha\ge 0} \left\{\sum_{U \subset E} \rank_{1}(U) w_{U} +
    \rank_{1}(E) \mu \mid \sum_{U \subset E \colon e \in U} w_{U} + \mu \leq
    c_{1}(e) - \alpha_{e} \ \forall e\in E,  w_{U} \leq 0 \forall U \subset E \right\}  
    \\ =
    \min_{X\in \basis_1} c_1(X)-\alpha(X),
\end{multline}
    
\begin{multline} \max_{\beta\ge 0} \left\{\sum_{U \subset E} \rank_{2}(U) v_{U} +
    \rank_{2}(E) \nu \mid \sum_{U \subset E \colon e \in U} v_{U} + \nu \leq
    c_{2}(e) - \beta_{e} \ \forall e\in E,  v_{U} \leq 0 \forall U \subset E \right\} 
    \\= 
    \min_{Y\in \basis_2} c_2(Y)-\beta(Y).
\end{multline}

Thus, replacing the two inner problems by their respective duals, we can rewrite  $(D_{\lambda})$ as follows:
\begin{align*}
     \max\ & \left(\min_{X\in \basis_1} \left(c_1(X)-\alpha(X)\right) +
     \min_{Y\in \basis_2} \left(c_2(Y)-\beta(Y)\right)\right)& \\
    \text{s.t. } 
    & \alpha_{e} + \beta_{e} \geq \lambda & \forall e \in E & \\
    & \alpha_{e}, \beta_{e} \geq 0. & \forall e \in E & 
\end{align*}

Now, take any tuple $(X,Y, \alpha, \beta)$ satisfying the optimality conditions (i),(ii) and (iii) for $\lambda$. Observe that the incidence vectors $x$ and $y$ of $X$ and $Y$, respectively, together with the incidence vector $z$ of the intersection $X\cap Y$, is a feasible solution of the primal LP, while
$\alpha$ and $\beta$ yield a feasible solution of the dual LP.
Since 
$$   c_1(X)-\alpha(X)+c_2(Y) -\beta(Y)=c_1(X)+c_2(Y)-\sum_{e\in X\cap Y} (\alpha_e+\beta_e)=c_1(X)+c_2(Y)-\lambda |X\cap Y|,
$$
the objective values of the primal and dual feasible solutions coincide.
It follows that any tuple $(X,Y,\alpha, \beta, \lambda)$ satisfying optimality conditions (i),(ii) and (iii) must be optimal for val$(\lambda)$ and its dual. \qed

\end{document}